\documentclass[a4paper]{amsart}
\usepackage{amssymb}
\usepackage{latexsym}

\usepackage[latin1]{inputenc}
\usepackage[english]{babel}
\usepackage[bookmarksnumbered,plainpages,hypertex]{hyperref}

\theoremstyle{plain}
\newtheorem{theorem}{Theorem}[section]
\newtheorem{lemma}[theorem]{Lemma}
\newtheorem{proposition}[theorem]{Proposition}

\theoremstyle{definition}

\theoremstyle{remark}
\newtheorem{remark}[theorem]{Remark}

\newcommand{\ic}{\ensuremath{\mathcal{I}}}
\newcommand{\oc}{\ensuremath{\mathcal{O}}}

\newcommand{\ec}{\ensuremath{\mathcal{E}}}

\newcommand{\nc}{\ensuremath{\mathcal{N}}}

\newcommand{\lag}{\langle}
\newcommand{\rag}{\rangle}

\newcommand{\Pt}{\mathbb{P}^3}

\newcommand{\Pq}{\mathbb{P}^4}

\newcommand{\Pn}{\mathbb{P}^n}

\newcommand{\bZ}{\mathbb{Z}}

\begin{document}

\title[Rank two globally generated vector bundles with $c_1\leq 5$.]{Rank two globally generated vector bundles with $c_1\leq 5$.}

\author{Chiodera L.}

\address{Dipartimento di Matematica, via Machiavelli 35, 44100 Ferrara (Italy)}
\email{ludovica.chiodera@unife.it}


\author{Ellia Ph.}
\address{Dipartimento di Matematica, via Machiavelli 35, 44100 Ferrara (Italy)}
\email{phe@unife.it}

\maketitle

\thispagestyle{empty}

\begin{abstract} We classify globally generated rank two vector bundles on $\Pn$, $n \geq 3$, with $c_1 \leq 5$. The classification is complete but for one case ($n=3$, $c_1=5$, $c_2=12$).
\end{abstract}

\section*{Introduction.} Vector bundles generated by global sections are basic objects in projective algebraic geometry. Globally generated line bundles correspond to morphisms to a projective space, more generally higher rank bundles correspond to morphism to (higher) Grassmann varieties. For this last point of view (that won't be touched in this paper) see\cite{SU-0} \cite{SU}, \cite{Suk}. Also globally generated vector bundles appear in a variety of problems (\cite{FM} just to make a single, recent example).
\par 
In this paper we classify globally generated rank two vector bundles on $\Pn$ (projective space over $k$, $\overline k = k$, $ch(k)=0$), $n \geq 3$, with $c_1 \leq 5$. The result is:

\begin{theorem}
\label{ThmFin}
Let $E$ be a rank two vector bundle on $\Pn$, $n \geq 3$, generated by global sections with Chern classes $c_1,c_2$, $c_1 \leq 5$.
\begin{enumerate}
\item If $n \geq 4$, then $E$ is the direct sum of two line bundles
\item If $n=3$ and $E$ is indecomposable, then $$(c_1,c_2) \in S = \{((2,2),(4,5),(4,6),(4,7),(4,8),(5,8),(5,10),(5,12) \}.$$ If $E$ exists there is an exact sequence: $$0 \to \oc \to E \to \ic _C(c_1) \to 0\,\,(*)$$ where $C \subset \Pt$ is a smooth curve of degree $c_2$ with $\omega _C(4-c_1) \simeq \oc _C$. The curve $C$ is \emph{irreducible}, except maybe if $(c_1,c_2)=(4,8)$: in this case $C$ can be either irreducible or the disjoint union of two smooth conics.
\item For every $(c_1,c_2) \in S$, $(c_1,c_2) \neq (5,12)$, there exists a rank two vector bundle on $\Pt$ with Chern classes $(c_1,c_2)$ which is globally generated (and with an exact sequence as in (2)).
\end{enumerate}
\end{theorem}

The classification is complete, but for one case: we are unable to say if there exist or not globally generated rank two vector bundles with Chern classes $c_1 =5, c_2 =12$ on $\Pt$.


\section{Rank two vector bundles on $\Pt$.}

\subsection{General facts.}\quad

For completeness let's recall the following well known results:

\begin{lemma}
Let $E$ be a rank $r$ vector bundle on $\Pn$, $n \geq 3$. Assume $E$ is generated by global sections.
\begin{enumerate}
\item If $c_1(E)=0$, then $E \simeq r.\oc$
\item If $c_1(E) =1$, then $E \simeq \oc (1)\oplus (r-1).\oc$ or $E \simeq T(-1)\oplus (r-n).\oc$.
\end{enumerate}
\end{lemma}

\begin{proof} If $L \subset \Pn$ is a line then $E|L \simeq \bigoplus _{i=1}^r \oc _L(a_i)$ by a well known theorem and $a_i \geq 0, \forall i$ since $E$ is globally generated. It turns out that in both cases: $E|L \simeq \oc _L(c_1)\oplus (r-1).\oc _L$ for every line $L$, i.e. $E$ is uniform. Then (1) follows from a result of Van de Ven (\cite{VdV}), while (2) follows from IV. Prop. 2.2 of \cite{E}.
\end{proof}

\begin{lemma}
Let $E$ be a rank two vector bundle on $\Pn$, $n \geq 3$. If $E$ has a nowhere vanishing section then $E$ splits. If $E$ is generated by global sections and doesn't split then $h^0(E) \geq 3$ and a general section of $E$ vanishes along a smooth curve, $C$, of degree $c_2(E)$ such that $\omega _C(4-c_1) \simeq \oc _C$. Moreover $\ic _C(c_1)$ is generated by global sections.
\end{lemma}

\begin{lemma}
Let $E$ be a non split rank two vector bundle on $\Pt$ with $c_1=2$. If $E$ is generated by global sections then $E$ is a null-correlation bundle.
\end{lemma}

\begin{proof} We have an exact sequence: $0 \to \oc \to E \to \ic _C(2) \to 0$, where $C$ is a smooth curve with $\omega _C(2) \simeq \oc _C$. It follows that $C$ is a disjoint union of lines. Since $h^0(\ic _C(2))\geq 2$, $d(C) \leq 2$. Finally $d(C)=2$ because $E$ doesn't split.
\end{proof}

This settles the classification of rank two globally generated vector bundles with $c_1(E) \leq 2$ on $\Pt$.

\subsection{Globally generated rank two vector bundles with $c_1=3$.}\quad

The following result has been proved in \cite{Suk} (with a different and longer proof).

\begin{proposition}
\label{c1=3n=3}
Let $E$ be a rank two globally generated vector bundle on $\Pt$. If $c_1(E)=3$ then $E$ splits. 
\end{proposition}

\begin{proof}
Assume a general section vanishes in codimension two, then it vanishes along a smooth curve $C$ such that $\omega _C\simeq \oc _C (-1)$. Moreover $\ic _C(3)$ is generated by global sections. We have $C = \cup _{i=1}^rC_i$ (disjoint union) where each $C_i$ is smooth irreducible with $\omega _{C_i}\simeq \oc _{C_i}(-1)$. It follows that each $C_i$ is a smooth conic. If $r\geq 2$ let $L=\lag C_1\rag \cap \lag C_2\rag$ ($\lag C_i\rag$ is the plane spanned by $C_i$). Every cubic containing $C$ contains $L$ (because it contains the four points $C_1\cap L$, $C_2\cap L$). This contradicts the fact that $\ic _C(3)$ is globally generated. Hence $r=1$ and $E = \oc (1)\oplus \oc (2)$. 
\end{proof}

\subsection{Globally generated rank two vector bundles with $c_1 = 4$.}\quad

Let's start with a general result:

\begin{lemma}
\label{boundc2}
Let $E$ be a non split rank two vector bundle on $\Pt$ with Chern classes $c_1, c_2$. If $E$ is globally generated and if $c_1 \geq 4$ then:
$$c_2 \leq \frac{2c_1^3-4c_1^2+2}{3c_1-4}.$$
\end{lemma}

\begin{proof} By our assumptions a general section of $E$ vanishes along a smooth curve, $C$, such that $\ic _C(c_1)$ is generated by global sections. Let $U$ be the complete intersections of two general surfaces containing $C$. Then $U$ links $C$ to a smooth curve, $Y$. We have $Y \neq \varnothing$ since $E$ doesn't split. The exact sequence of liaison: $0 \to \ic _U(c_1) \to \ic _C(c_1) \to \omega _Y(4-c_1) \to 0$ shows that $\omega _Y(4-c_1)$ is generated by global sections. Hence $\deg (\omega _Y(4-c_1))\geq 0$. We have $\deg (\omega _Y(4-c_1))=2g' -2 +d'(4-c_1)$ ($g' =p_a(Y)$, $d'=\deg (Y)$). So $g' \geq \frac{d'(c_1-4)+2}{2}\geq 0$ (because $c_1 \geq 4$). On the other hand, always by liaison, we have: $g'-g= \frac{1}{2}(d'-d)(2c_1-4)$ ($g=p_a(C)$, $d=\deg (C)$). Since $d'=c_1^2-d$ and $g=\frac{d(c_1-4)}{2}+1$ (because $\omega _C(4-c_1) \simeq \oc _C$), we get: $g' =1+\frac{d(c_1-4)}{2}+\frac{1}{2}(c_1^2-2d)(2c_1-4) \geq 0$ and the result follows.
\end{proof} 

\par

Now we have:

\begin{proposition}
\label{n=3c1=4}
Let $E$ be a rank two globally generated vector bundle on $\Pt$. If $c_1(E)=4$ and if $E$ doesn't split, then $5 \leq c_2 \leq 8$ and there is an exact sequence:
$0 \to \oc \to E \to \ic _C(4) \to 0$, where $C$ is a smooth irreducible elliptic curve of degree $c_2$ or, if $c_2=8$, $C$ is the disjoint union of two smooth elliptic quartic curves.
\end{proposition}

\begin{proof}
A general section of $E$ vanishes along $C$ where $C$ is a smooth curve with $\omega _C = \oc _C$ and where $\ic _C(4)$ is generated by global sections. Let $C=C_1\cup ...\cup C_r$ be the decomposition into irreducible components: the union is disjoint, each $C_i$ is a smooth elliptic curve hence has degree at least three.\\
By Lemma \ref{boundc2} $d = \deg (C) \leq 8$. If $d \leq 4$ then $C$ is irreducible and is a complete intersection which is impossible since $E$ doesn't split. If $d=5$, $C$ is smooth irreducible.\\
{\it Claim:} If $8 \geq d \geq 6$, $C$ cannot contain a plane cubic curve.\\
Assume $C=P\cup X$ where $P$ is a plane cubic and where $X$ is a smooth elliptic curve of degree $d-3$. If $d=6$, $X$ is also a plane cubic and every quartic containing $C$ contains the line $\lag P\rag \cap \lag X \rag$. If $deg(X)\geq 4$ then every quartic, $F$, containing $C$ contains the plane $\lag P\rag$. Indeed $F|H$ vanishes on $P$ and on the $deg(X) \geq 4$ points of $X \cap \lag P \rag$, but these points are not on a line so $F|H=0$. In both cases we get a contradiction with the fact that $\ic _C(4)$ is generated by global sections. The claim is proved.\\
It follows that, if $8 \geq d \geq 6$, then $C$ is irreducible except if $C=X \cup Y$ is the disjoint union of two elliptic quartic curves.
\end{proof}

\par
Now let's show that all possibilities of Proposition \ref{n=3c1=4} do actually occur. For this we have to show the existence of a smooth irreducible elliptic curve of degree $d$, $5 \leq d \leq 8$ with $\ic _C(4)$ generated by global sections (and also that the disjoint union of two elliptic quartc curves is cut off by quartics).

\begin{lemma}
\label{c1=4c2=5}
There exist rank two vector bundles with $c_1=4, c_2=5$ which are globally generated. More precisely any such bundle is of the form $\nc (2)$, where $\nc$ is a null-correlation bundle (a stable bundle with $c_1=0, c_2=1$).
\end{lemma}

\begin{proof} The existence is clear (if $\nc$ is a null-correlation bundle then it is well known that $\nc (k)$ is globally generated if $k\geq 1$). Conversely if $E$ has $c_1=4, c_2=5$ and is globally generated, then $E$ has a section vanishing along a smooth, irreducible quintic elliptic curve (cf \ref{n=3c1=4}). Since $h^0(\ic _C(2))=0$, $E$ is stable, hence $E = \nc (2)$.
\end{proof}

\begin{lemma}
\label{c1=4c2=6}
There exist smooth, irreducible elliptic curves, $C$, of degree $6$ with $\ic _C(4)$ generated by global sections.
\end{lemma}

\begin{proof} Let $X$ be the union of three skew lines. The curve $X$ lies on a smooth quadric surface, $Q$, and has $\ic _X(3)$ globally generated (indeed the exact sequence $0 \to \ic _Q \to \ic _X \to \ic _{X,Q} \to 0$ twisted by $\oc (3)$ reads like: $0 \to \oc (1) \to \ic _C(3) \to \oc _Q(3,0) \to 0$). The complete intersection, $U$, of two general cubics containing $X$ links $X$ to a smooth curve, $C$, of degree $6$ and arithmetic genus $1$. Since, by liaison, $h^1(\ic _C)=h^1(\ic _X(-2))=0$, $C$ is irreducible. The exact sequence of liaison: $0 \to \ic _U(4) \to \ic _C(4) \to \omega _X(2) \to 0$ shows that $\ic _C(4)$ is globally generated.
\end{proof}

\par
In order to prove the existence of smooth, irreducible elliptic curves, $C$, of degree $d=7,8$, with $\ic _C(4)$ globally generated, we have to recall some results due to Mori (\cite{M}). 
\par
According to \cite{M} Remark 4, Prop. 6, there exists a smooth quartic surface $S \subset \Pt$ such that $Pic(S) = \bZ H \oplus \bZ X$ where $X$ is a smooth elliptic curve of degree $d$ ($7 \leq d \leq 8$). The intersection pairing is given by: $H^2=4$, $X^2=0$, $H.X=d$. Such a surface doesn't contain any smooth rational curve (\cite{M} p.130). In particular: $(*)$ every integral curve, $Z$, on $S$ has degree $\geq 4$ with equality if and only if $Z$ is a planar quartic curve or an elliptic quartic curve.

\begin{lemma}
\label{notOnCubic}
With notations as above, $h^0(\ic _X(3))=0$.
\end{lemma}

\begin{proof} A curve $Z \in |3H - X|$ has invariants $(d_Z, g_Z)= (5,-2)$ (if $d=7$) or $(4,-5)$ (if $d=8$), so $Z$ is not integral. It follows that $Z$ must contain an integral curve of degree $<4$, but this is impossible.
\end{proof}

\begin{lemma}
\label{d=7,8-g=1}
With notations as above $|4H -X|$ is base point free, hence there exist smooth, irreducible elliptic curves, $X$, of degree $d$, $7 \leq d \leq 8$, such that $\ic _X(4)$ is globally generated.
\end{lemma}

\begin{proof} Let's first prove the following:
{\it Claim:} Every curve in $|4H - X|$ is integral.
\par
If $Y \in |4H -X|$ is not integral then $Y = Y_1+Y_2$ where $Y_1$ is integral with $\deg (Y_1)=4$ (observe that $\deg (Y) = 9$ or $8$). 
\par
If $Y_1$ is planar then $Y_1 \sim H$, so $4H - X \sim H+Y_2$ and it follows that $3H \sim X+Y_2$, in contradiction with $h^0(\ic _X(3))=0$ (cf \ref{notOnCubic}).
\par
So we may assume that $Y_1$ is a quartic elliptic curve, i.e. (i) $Y_1^2=0$ and (ii) $Y_1.H=4$. Setting $Y_1 = aH+bX$, we get from (i): $2a(2a+bd)=0$. Hence $(\alpha )\,\, a =0$, or $(\beta )\,\, 2a+bd=0$. 
\par
$(\alpha )$ In this case $Y_1=bX$, hence (for degree reasons and since $S$ doesn't contain curves of degree $<4$), $Y_2 = \varnothing$ and $Y =X$, which is integral.
\par
$(\beta )$ Since $Y_1.H=4$, we get $2a+(2a+bd)=2a=4$, hence $a=2$ and $bd =-4$ which is impossible ($d=7$ or $8$ and $b \in \bZ$).
\par
This concludes the proof of the claim.
\par
Since $(4H -X)^2 \geq 0$, the claim implies that $4H - X$ is numerically effective. Now we conclude by a result of Saint-Donat (cf \cite{M}, Theorem 5) that $|4H - X|$ is base point free, i.e. $\ic _{X,S}(4)$ is globally generated. By the exact sequence: $0 \to \oc \to \ic _X(4) \to \ic _{X,S}(4) \to 0$ we get that $\ic _X(4)$ is globally generated.
\end{proof}

\begin{remark}
If $d=8$, a general element $Y \in |4H - X|$ is a smooth elliptic curve of degree $8$. By the way $Y \neq X$ (see \cite{BE}). The exact sequence of liaison: $0 \to \ic _U(4) \to \ic _X(4) \to \omega _Y \to 0$ shows that $h^0(\ic _X(4))=3$ (i.e. $X$ is of maximal rank).
In case $d=8$ Lemma \ref{d=7,8-g=1} is stated in \cite{D}, however the proof there is incomplete, indeed in order to apply the enumerative formula of \cite{F} one has to know that $X$ is a \emph{connected} component of $\displaystyle \bigcap _{i=1}^3 F_i$; this amounts to say that the base locus of $|4H-X|$ on $F_1$ has dimension $\leq 0$. 
\end{remark}

\par
To conclude we have:

\begin{lemma}
Let $X$ be the disjoint union of two smooth, irreducible quartic elliptic cuvres, then $\ic _X(4)$ is generated by global sections.
\end{lemma}

\begin{proof} Let $X = C_1 \sqcup C_2$. We have: $0 \to \oc (-4) \to 2.\oc (-2) \to \ic _{C_1} \to 0$, twisting by $\ic _{C_2}$, since $C_1 \cap C_2 = \varnothing$, we get:\\ $0 \to \ic _{C_2}(-4) \to 2.\ic _{C_2}(-2) \to \ic _X \to 0$ and the result follows.
\end{proof}

Summarizing:

\begin{proposition}
\label{c1=4Fin}
There exists an indecomposable rank two vector bundle, $E$, on $\Pt$, generated by global sections and with $c_1(E)=4$ if and only if $5 \leq c_2(E) \leq 8$ and in these cases there is an exact sequence:
$$0 \to \oc \to E \to \ic _C(4) \to 0$$
where $C$ is a smooth irreducible elliptic curve of degree $c_2(E)$ or, if $c_2(E) =8$, the disjoint union of two smooth elliptic quartic curves.
\end{proposition}

\subsection{Globally generated rank two vector bundles with $c_1 = 5$.}\quad

We start by listing the possible cases:

\begin{proposition}
\label{possiblec1=5}
If $E$ is an indecomposable, globally generated, rank two vector bundle on $\Pt$ with $c_1(E)=5$, then $c_2(E) \in \{8,10, 12 \}$ and there is an exact sequence:
$$0 \to \oc \to E \to \ic _C(5) \to 0$$
where $C$ is a smooth, irreducible curve of degree $d=c_2(E)$, with $\omega _C \simeq \oc _C(1)$.
\par
In any case $E$ is stable.
\end{proposition}

\begin{proof} A general section of $E$ vanishes along a smooth curve, $C$, of degree $d = c_2(E)$ with $\omega _C \simeq \oc _C(1)$. Hence every irreducible component, $Y$, of $C$ is a smooth, irreducible curve with $\omega _Y \simeq \oc _Y(1)$. In particular $\deg (Y)=2g(Y)-2$ is even and $\deg (Y) \geq 4$.
\begin{enumerate}
\item If $d=4$, then $C$ is a planar curve and $E$ splits.
\item If $d=6$, $C$ is necessarily irreducible (of genus $4$). It is well known that any such curve is a complete intersection $(2,3)$, hence $E$ splits.
\item If $d=8$ and $C$ is not irreducible, then $C = P_1 \sqcup P_2$, the disjoint union of two planar quartic curves. If $L = \lag P_1 \rag \cap \lag P_2 \rag$, then every quintic containing $C$ contains $L$ in contradiction with the fact that $\ic _C(5)$ is generated by global sections. Hence $C$ is irreducible.
\item If $d=10$ and $C$ is not irreducible, then $C = P \sqcup X$, where $P$ is a planar curve of degree $4$ and where $X$ is a degree $6$ curve ($X$ is a complete intersection $(2,3)$). Every quintic containing $C$ vanishes on $P$ and on the $8$ points of $X \cap \lag P \rag$, since these $8$ points are not on a line, the quintic vanishes on the plane $\lag P \rag$. This contradicts the fact that $\ic _C(5)$ is globally generated.
\item If $d=12$ and $C$ is not irreducible we have three possibilities:
\begin{enumerate}
\item $C = P_1 \sqcup P_2 \sqcup P_3$, $P_i$ planar quartic curves
\item $C = X_1 \sqcup X_2$, $X_i$ complete intersection curves of types $(2,3)$
\item $C = Y \sqcup P$, $Y$ a canonical curve of degree $8$, $P$ a planar curve of degree $4$.
\end{enumerate}
(a) This case is impossible (consider the line $\lag P_1 \rag \cap \lag P_2 \rag$).\\
(b) We have $X_i = Q_i \cap F_i$. Let $Z$ be the quartic curve $Q_1 \cap Q_2$. Then $X_i \cap Z = F_i \cap Z$, i.e. $X_i$ meets $Z$ in $12$ points. It follows that every quintic containing $C$ meets $Z$ in $24$ points, hence such a quintic contains $Z$. Again this contradicts the fact that $\ic _C(5)$ is globally generated.\\
(c) This case too is impossible: every quintic containing $C$ vanishes on $P$ and on the points $\lag P \rag \cap Y$, hence on $\lag P \rag$.
\par
We conclude that if $d=12$, $C$ is irreducible.
\end{enumerate}
\par
The normalized bundle is $E(-3)$, since in any case $h^0(\ic _C(2))=0$ (every smooth irreducible subcanonical curve on a quadric surface is a complete intersection), $E$ is stable.
\end{proof}

\par
Now we turn to the existence part.

\begin{lemma}
\label{Ex-c1=5}
There exist indecomposable rank two vector bundles on $\Pt$ with Chern classes $c_1=5$ and $c_2 \in \{8,10 \}$ which are globally generated.
\end{lemma}

\begin{proof} Let $R = \sqcup _{i=1}^s L_i$ be the union of $s$ disjoint lines, $2 \leq s \leq 3$. We may perform a liaison $(s,3)$ and link $R$ to $K = \sqcup _{i=1}^s K_i$, the union of $s$ disjoint conics. The exact sequence of liaison: $0 \to \ic _U(4) \to \ic _K(4) \to \omega _R(5-s) \to 0$ shows that $\ic _K(4)$ is globally generated (n.b. $5-s \geq 2$).
\par \noindent
Since $\omega _K(1) \simeq \oc _K$ we have an exact sequence: $0 \to \oc \to \ec (2) \to \ic _K(3) \to 0$, where $\ec$ is a rank two vector bundle with Chern classes $c_1 = -1, c_2 = 2s-2$. Twisting by $\oc (1)$ we get: $0 \to \oc (1) \to \ec (3) \to \ic _K(4) \to 0\,\,\,(*)$. The Chern classes of $\ec (3)$ are $c_1=5$, $c_2=2s+4$ (i.e. $c_2=8, 10$). Since $\ic _K(4)$ is globally generated, it follows from $(*)$ that $\ec (3)$ too, is generated by global sections.
\end{proof}

\begin{remark}\quad
\label{rmk-Exc1=5}
\begin{enumerate}
\item If $\ec$ is as in the proof of Lemma \ref{Ex-c1=5} a general section of $\ec (3)$ vanishes along a smooth, irreducible (because $h^1(\ec (-2))=0$) canonical curve, $C$, of genus $1+c_2/2$ ($g=5,6$) such that $\ic _C(5)$ is globally generated. By construction these curves are not of maximal rank ($h^0(\ic _C(3))=1$ if $g=5$, $h^0(\ic _C(4))=2$ if $g=6$). As explained in \cite{GP} §4 this is a general fact: no canonical curve of genus $g, 5\leq g \leq 6$ in $\Pt$ is of maximal rank. We don't know if this is still true for $g=7$.
\item According to \cite{GP} the general canonical curve of genus $6$ lies on a unique quartic surface.
\item The proof of \ref{Ex-c1=5} breaks down with four conics: $\ic _K(4)$ is no longer globally generated, every quartic containing $K$ vanishes along the lines $L_i$ ($5-s=1$). Observe also that four disjoint lines always have a quadrisecant and hence are an exception to the \emph{normal generation conjecture}(the omogeneous ideal is not generated in degree three as it should be).
\end{enumerate}
\end{remark}

\begin{remark}\quad
\label{c15c212}
The case $(c_1,c_2)=(5,12)$ remains open. It can be shown that if $E$ exists, a general section of $E$ is linked, by a complete intersections of two quintics, to a smooth, irreducible curve, $X$, of degree $13$, genus $10$  having $\omega _X(-1)$ as a base point free $g^1_5$. One can prove the existence of curves $X \subset \Pt$, smooth, irreducible, of degree $13$, genus $10$, with $\omega _X(-1)$ a base point free pencil and lying on \emph{one} quintic surface. But we are unable to show the existence of such a curve with $h^0(\ic _X(5))\geq 3$ (or even with $h^0(\ic _X(5))\geq 2$). We believe that such bundles do not exist.
\end{remark}

\section{Globally generated rank two vector bundles on $\Pn$, $n \geq 4$.}

For $n \geq 4$ and $c_1 \leq 5$ there is no surprise:

\begin{proposition}
\label{n>3}
Let $E$ be a globally generated rank two vector bundle on $\Pn$, $n \geq 4$. If $c_1(E) \leq 5$, then $E$ splits.
\end{proposition}

\begin{proof}
It is enough to treat the case $n=4$. A general section of $E$ vanishes along a smooth (irreducible)  subcanonical surface, $S$:
$0 \to \oc \to E \to \ic _S(c_1) \to 0$. By \cite{EFG}, if $c_1 \leq 4$, then $S$ is a complete intersection and $E$ splits. Assume now $c_1=5$. Consider the restriction of $E$ to a general hyperplane $H$. If $E$ doesn't split, by \ref{possiblec1=5} we get that the normalized Chern classes of $E$ are: $c_1=-1$, $c_2\in\{2,4,6\}$. By Schwarzenberger condition: $c_2(c_2+2)\equiv 0\,\,(mod\,\, 12)$. The only possibilities are $c_2=4$ or $c_2=6$. If $c_2 =4$, since $E$ is stable (because $E|H$ is, see \ref{possiblec1=5}), we have (\cite{De}) that $E$ is a Horrocks-Mumford bundle. But the Horrocks-Mumford bundle (with $c_1=5$) is not globally generated.
\par
The case $c_2=6$ is impossible: such a bundle would yield a smooth surface $S \subset \Pq$, of degree $12$ with $\omega _S \simeq \oc _S$, but the only smooth surface with $\omega _S \simeq \oc _S$ in $\Pq$ is the abelian surface of degree $10$ of Horrocks-Mumford.
\end{proof}

\begin{remark}\quad
For $n>4$ the results in \cite{EFG2} give stronger and stronger (as $n$ increases) conditions for the existence of indecomposable rank two vector bundles generated by global sections.
\end{remark}

Putting everything together, the proof of Theorem \ref{ThmFin} is complete.


\end{document}